\documentclass[10pt,a4paper]{amsart}

\usepackage[all]{xy}
\usepackage[normalem]{ulem}
\usepackage{amsmath}
\usepackage{amsfonts}
\usepackage{amssymb}
\usepackage{amsthm}
\usepackage{graphicx}
\usepackage{microtype}
\usepackage[usenames,dvipsnames]{xcolor}
\usepackage{pgf,tikz}
\usepackage{mathrsfs}
\usetikzlibrary{arrows}
\usepackage{caption}
\usepackage{soul}
\usepackage[pdftex,plainpages=false,bookmarks=true,breaklinks=true,linkcolor=brown,citecolor=brown,colorlinks]{hyperref}
\usepackage{bm}
\usepackage{cancel} 
\usepackage{combelow}
\usepackage{here}

\newcommand{\prarrow}[2]{\ar@<0.5ex>[r]^-{#1} \ar@<-0.5ex>[r]_-{#2}}
\newcommand{\plarrow}[2]{\ar@<0.5ex>[l]^-{#1} \ar@<-0.5ex>[l]_-{#2}}
\newcommand{\pdarrow}[2]{\ar@<0.5ex>[d]^-{#1} \ar@<-0.5ex>[d]_-{#2}}
\newcommand{\puarrow}[2]{\ar@<0.5ex>[u]^-{#1} \ar@<-0.5ex>[u]_-{#2}}

\newtheorem{theorem}{Theorem}[section]
\newtheorem{corollary}[theorem]{Corollary}    
\newtheorem{lemma}[theorem]{Lemma}

\theoremstyle{definition}
\newtheorem{definition}[theorem]{Definition} \newtheorem{question}{Question}  
\newtheorem{example}[theorem]{Example}    
\newtheorem{remark}[theorem]{Remark}

\author{Hayato Imamura}
\address{Faculty of Economics, The International University of Kagoshima, 8-34-1 Sakanoue, Kagoshima-shi, Kagoshima, 891-0197, Japan.}
\email{hayato-imamura@asagi.waseda.jp}

\author{Eiichi Matsuhashi}
\address{Department of Mathematics, Shimane University, Matsue, Shimane, 690-8504, Japan.}
\email{matsuhashi@riko.shimane-u.ac.jp}
\title{Some theorems on decomposable continua}

\author{Yoshiyuki Oshima}
\address{Department of Mathematics, Shimane University, Matsue, Shimane, 690-8504, Japan.}
\email{oshima@riko.shimane-u.ac.jp}

\begin{document}

\begin{abstract}
We prove some theorems on decomposable continua. In particular, we prove;

(i) the property of being a Wilder continuum is not a Whitney reversible property,

(ii)  inverse limits of $D^{**}$-continua with surjective monotone upper semi-continuous bonding functions are $D^{**}$, and

(iii) there exists a $D^{**}$-continuum which contains neither Wilder continua nor $D^{*}$-continua. 

\vspace{1mm}

Also, we show the existence of a Wilder continuum containing no $D^*$-continua and a $D^*$-continuum  containing no Wilder continua.

\end{abstract}


\maketitle
\markboth{}{Hayato Imamura, Eiichi Matsuhashi and Yoshiyuki Oshima}

\section{Introduction}
    In this paper, all spaces are assumed to be metrizable. A $map$ is a continuous single-valued function. A map $f:X\to Y$ is said to be $monotone$ if for each $y\in Y$, $f^{-1}(y)$ is connected.
 If $X$ is a space and $A$ is a subset of $X$, then we denote the closure of $A$ in  $X$ by ${\rm Cl}_X A$. 
 By a $continuum$ we mean a compact connected metric space. An $arc$ is a space which is homeomorphic to the closed interval $[0,1]$. If $A$ is an arc, then we denote the set of end points of $A$ by $E(A)$. A continuum
is said to be $decomposable$ if it is  the union of two proper subcontinua. If a continuum $X$ is not decomposable, then $X$ is said to be $indecomposable$.

Let $X$ be
a continuum.  Then, $2^X$ denotes the space of all closed subsets of $X$ with the topology generated by the Hausdorff metric. Also, $C(X)$ denotes the  subspace of $2^X$ consisting of all nonempty subcontinua of $X$.  
A $Whitney$
$map$  is a map $\mu : C(X) \to [0,\mu(X))$ satisfying  $\mu(\{x\})=0$ for each $x \in X$, and $\mu(A) < \mu(B)$  whenever $A, B \in C(X)$ and $A \subsetneq B$. It is well-known  that for each Whitney map  $\mu: C(X) \to [0,\mu(X)]$   
and for each $t \in [0, \mu(X)]$,  $\mu^{-1}(t)$ is a continuum (\cite[Theorem 19.9]{illanes}). 

 A topological property $P$ is called a {\it Whitney property} if a continuum $X$ has property $P$, so does $\mu^{-1}(t)$ for any Whitney map  $\mu$ for $C(X)$ and for any $t \in [0, \mu(X))$.   Also, a topological property $P$ is called a {\it Whitney reversible property} provided that whenever $X$ is a continuum such that $\mu^{-1}(t)$ has property $P$ for each Whitney map $\mu$ for $C(X)$ and for each $t \in (0, \mu(X))$, then $X$ has property $P$. 
 For  information about Whitney properties and Whitney reversible properties, see \cite[Chapter 8]{illanes}.

In this paper, we show some theorems on decomposable continua. In particular, we deal with topics on Wilder continua, $D$-continua, $D^{*}$-continua and $D^{**}$-continua (definitions of these continua are given in the following section). Studies on these continua are seen in  \cite{wwp}, \cite{D}, \cite{loncar1}, \cite{loncar2}, \cite{loncar3}, \cite{loncar4}, \cite{D**} and \cite{wilder2}. 

\vspace{1mm}

The paper is divided as follows:

In Section 2, we introduce some basic definitions and terminology that are used throughout this paper. 

In Section 3, we show that the properties of being a Wilder continuum is not a Whitney reversible property. This result gives a partial answer to \cite[Qusetion 5.2]{D**}.  Note that the properties of being a Wilder continuum, being a $D$-continuum, being a $D^{*}$-continuum and being a $D^{**}$-continuum are all Whitney properties (see \cite{D**}). 

In Section 4, we deal with topics on inverse  limits with set-valued bonding functions. The notion of inverse limits with set-valued bonding functions was introduced by Ingram and Mahavier (\cite{ingram-mahavier}). 
In  Section 4,  we prove that inverse limits of $D^{**}$-continua with surjective monotone upper semi-continuous bonding functions are $D^{**}$. This result is in contrast to \cite[Examples 5.1 and 5.2]{oshima},  which state that there exists an inverse sequence $\{X_i,f_i\}_{i=1}^{\infty}$ of $D$-continua (resp. $D^*$-continua) with surjective monotone upper semi-continuous bonding functions such that $\underleftarrow{{\rm lim}}\{X_i,f_i\}_{i=1}^{\infty}$ is not a $D$-continuum (resp. a $D^*$-continuum).  

In Section 5, we show that  there exists a $D^{**}$-continuum which contains neither Wilder continua nor $D^{*}$-continua. Also, we show the existence of a Wilder continuum containing no $D^*$-continua and a $D^*$-continuum  containing no Wilder continua.

 \section{preliminaries}

In this section, we introduce some basic definitions and terminology.

 \begin{definition}
 A continuum $X$ is called a {\it Wilder continuum} if for any three  distinct points $x,y,z \in X$, there exists a subcontinuum $C$ of  $X$ such that $x \in C$ and  $C$ contains exactly one of $y$ and $z$.  
     
 \end{definition}

 Wilder introduced the notion of Wilder continua in \cite{wilder2}.\footnote{Wilder  named the continua as $C$-continua. However, we use the term Wilder continua following the notion proposed in \cite{kk}.}  

 \vspace{1mm}

 Let  $X$ be a continuum. If $X$ is not Wilder, then there exist three distinct points $x,y,z \in X$ such that for each subcontinuum $C$ of $X$ with $x \in C$ and $\{y,z\} \cap C \neq \emptyset$, $\{y,z\} \subseteq C$.  In this case, we say that $y$ and $z$ are $points$ $of$ $non$-$Wilderness$ $of$ $X$.

\begin{definition}
Let $X$ be a continuum. Then: 
\begin{itemize}
\item The continuum $X$ is called a {\it D-continuum} if for any pairwise disjoint nondegenerate subcontinua $A$ and $B$ of $X$, there exists a subcontinuum $C$ of $X$ such that 

(i)  $A \cap C \neq \emptyset \neq B \cap C$, and

(ii) $A \setminus C \neq \emptyset$ or $B \setminus C \neq \emptyset$.

\vspace{1mm}

\item The continuum $X$ is called a {\it $D^{**}$-continuum} if for any pairwise disjoint nondegenerate subcontinua $A$  and  $B$ of $X$, there exists a subcontinuum $C$ of $X$ such that 

(i)  $A \cap C \neq \emptyset \neq B \cap C$, and

(ii) $B \setminus C \neq \emptyset$.

\vspace{1mm}

\item The continuum $X$ is called a {\it $D^*$-continuum} if for any pairwise disjoint nondegenerate subcontinua $A$ and $B$ of $X$, there exists a subcontinuum $C$ of $X$ such that 

(i)  $A \cap C \neq \emptyset \neq B \cap C$, and

(ii) $A \setminus C \neq \emptyset$ and $B \setminus C \neq \emptyset$.
\end{itemize}

\end{definition}

Note that a continuum $X$ is a $D^{**}$-continuum if and only if for each $a \in X$ and for each nondegenerate subcontinuum $B$ of $X$ with $a \notin B$, there exists a subcontinuum $C$ of $X$ such that $a \in C$, $B \cap C \neq \emptyset$ and $B \setminus C \neq \emptyset$.


It is clear that every $D^*$-continuum is a $D^{**}$-continuum and every $D^{**}$-continuum   is a $D$-continuum. Note that the $\sin \frac{1}{x}$-continuum is a $D$-continuum which is not a $D^{**}$-continuum.  

In \cite[Section 2]{D**}, it is proven that  the classes of Wilder continua and $D^*$-continua are strictly contained in the class of $D^{**}$-continua.   Since every nondegenerate indecomposable continuum has uncountably many composants \cite[Theorem 11.15]{nadler1}, we can  easily  see that every nondegenerate $D$-continuum is decomposable. For other relationships between the classes of the above  continua and other classes of continua, see \cite[Figure 6]{D}. 

\smallskip

Let $X$ and $Y$ be any compacta. Then, a set-valued function $f: X\to 2^Y$ is called an $upper$ $semi$-$continuous$ $function$ provided that for each $x\in X$ and each open subset $V$ of $Y$ containing $f(x)$, there exists an open subset $U$ of $X$ which contains $x$ such that if $z\in U$, then $f(z)\subset V$. The $graph$ of a set-valued function $f:X\to 2^Y$ is the set $G(f)=\{(x,y)\in X\times Y~|~y\in f(x)\}$. We say that a set-valued function $f:X\to 2^Y$ is $surjective$ if for each $y\in Y$, there exists $x\in X$  such that $y\in f(x)$.

Every map $f:X\to Y$ induces the upper semi-continuous function $\tilde{f}:X\to 2^Y$ defined by $\tilde{f}(x)=\{f(x)\}$ for each $x\in X$. Therefore, if necessary, we may regard a map $f:X\to Y$ as the upper semi-continuous function $\tilde{f}:X\to 2^Y$.

An $inverse$ $sequence$ is a double sequence $\{X_i,f_i\}_{i=1}^{\infty}$ of compacta $X_i$ and upper semi-continuous functions $f_i:X_{i+1}\to 2^{X_i}$. The spaces $X_i$ are called the $factor$ $spaces$ and the upper semi-continuous functions $f_i$ the $bonding$ $functions$. We denote the subspace $\{(x_i)_{i=1}^{\infty}\in \prod_{i=1}^{\infty}X_i~|~$for each $i\geq 1,~x_i\in f_i(x_{i+1})\}$ of $\prod_{i=1}^{\infty}X_i$ by $\underleftarrow{{\rm lim}}\{X_i,f_i\}_{i=1}^{\infty}$. $\underleftarrow{{\rm lim}}\{X_i,f_i\}_{i=1}^{\infty}$ is  also denoted by $X_{\infty}$. $X_{\infty}$ is called the $inverse$ $limit$ $of$ $\{X_i,f_i\}_{i=1}^{\infty}$. 

\section{Being a Wilder continuum is not a  Whitney reversible property}

First, we show the following result which can yield a lot of examples of non-Wilder continua whose all  positive Whitney levels are Wilder continua.

\begin{theorem}
Let $X$ be a nondegenerate continuum and let $\mu: C(X) \to [0,\mu(X)]$ be a Whitney map. Assume there exists a totally disconnected subset  $T$  of $X$ satisfying the following property:

\smallskip

$\bullet$ If $a,b$ and $c$ are pairwise distinct points in $X \setminus T$, then there exists a subcontinuum $L$ of $X$ such that $a \in L$ and  $L$ contains exactly one of $b$ and $c$.

\smallskip

 Then, $\mu^{-1}(t)$ is a Wilder continuum for each $t \in (0,\mu(X))$.

\label{whitneywilder}
\end{theorem}

\begin{proof} The proof is obtained by slightly modifying the proof of  \cite[Theorem 4.2]{D**} as follows:

\smallskip

($\sharp$) In Case 3 of the  proof of  \cite[Theorem 4.2]{D**}, the second author and the third author chose points  $a \in A \setminus (B\cup C)$, $b \in B \setminus (A\cup C)$ and $c \in C \setminus (A\cup B)$. Instead of it, choose points $a \in A \setminus (B\cup C \cup T)$, $b \in B \setminus (A\cup C \cup T)$ and $c \in C \setminus (A\cup B \cup T)$.

\smallskip

Then, we can prove the theorem. \end{proof}

By the following example, we see that the property of being  a Wilder continuum is not a  Whitney reversible property. This result gives a partial answer to \cite[Question 5.2]{D**}. 

\begin{example} \label{nonwilder}
Let
$$
\begin{array}{ccl}
    X & = & \{(t+|{\rm cos}(\frac{1}{t})|,{\rm sin}(\frac{1}{t}))\in \mathbb{R}^2~|~0<t\leq 1\}\\
      &   & \cup \{({\rm cos}(t),{\rm sin}(t))\in \mathbb{R}^2~|~0\leq t\leq 2\pi\}\\
      &   & \cup \{(-t-|{\rm cos}(\frac{1}{t})|,{\rm sin}(\frac{1}{t}))\in \mathbb{R}^2~|~0<t\leq 1\}\\
      &   & \cup (\{-1-{\rm cos}(1)\}\times [-2,{\rm sin}(1)])\\
      &   & \cup ([-1-{\rm cos}(1),1+{\rm cos}(1)]\times \{-2\})\\
      &   & \cup (\{1+{\rm cos}(1)\}\times [-2,{\rm sin}(1)])\\
\end{array}
$$
(see  Figure \ref{figure1}). $X$ has two arc components. One of them is  a simple closed curve, and the other is homeomorphic to $\mathbb{R}$. Let $x$, $y$ and $z$ be points in $X$ as in  Figure \ref{figure1}. Since there does not exist a subcontinuum of $X$ which contains $x$ and exactly one of $y$ and $z$, $X$ is not Wilder. 

Let $\mu:C(X) \to [0, \mu(X)]$ be a Whitney map. Let $T=\{y,z\}$. Then, it is easy to see that $T$ is totally disconnected. Also, we see that if $a,b$ and $c$ are pairwise distinct points in $X \setminus T$, then there exists a subcontinuum $L$ of $X$ such that $a \in L$ and  $L$ contains exactly one of $b$ and $c$. Hence, by the previous result, we see that $\mu^{-1}(t)$ is a Wilder continuum for each $t \in (0,\mu(X))$.






    


\label{wilderexample}
\end{example}

\begin{remark}
In \cite[Example 3.5]{D},  Espinoza and the second author constructed an example of a $D^{*}$-continuum which is not Wilder. Compared to their example, it is easy to see that the continuum $X$ in the previous example has such properties.

\end{remark}

\begin{remark}
Let $X$, $y$, $z$ be as in Example \ref{nonwilder}. Let $c$ be the midpoint of the line segment  $yz$.   Let $n \ge 1$ and let $X_k$ be the subcontinuum of $\mathbb{R}^2$ with $X$ rotated $\frac{k \pi}{2n}$  clockwise around $c$  for each $1 \le k \le n$. Let $Z_0=X \times \{0\} \subseteq \mathbb{R}^3$ and let $Z_k=X_k \times \{\frac{1}{k}\} \subseteq \mathbb{R}^3$ for each $ 1 \le k \le n$. Now, let $S$ be the arc component of $X$ which is a simple closed curve, and for each $0 \le k \le n$, let $S_k= S \times \{\frac{1}{k}\} \subseteq Z_k$. Let  $\mathcal{D}$ be the upper semi-continuous decomposition of $\bigcup _{k=0}^n Z_k$ defined by $\mathcal{D}=\{\{a\} \times \{0, {1}, \frac{1}{2}, \ldots, \frac{1}{n}\} \ | \ a \in S\} \cup \{\{z\} \ | \ z \in \bigcup _{k=0}^n (Z_k \setminus  S_k)\}$. Then, by similar arguments in Example \ref{nonwilder}, we can easily see that the quotient space $X / \mathcal{D}$ is a non-Wilder continuum whose all positive Whitney levels are Wilder continua.


\end{remark}

\begin{figure}

\includegraphics[scale=0.12]{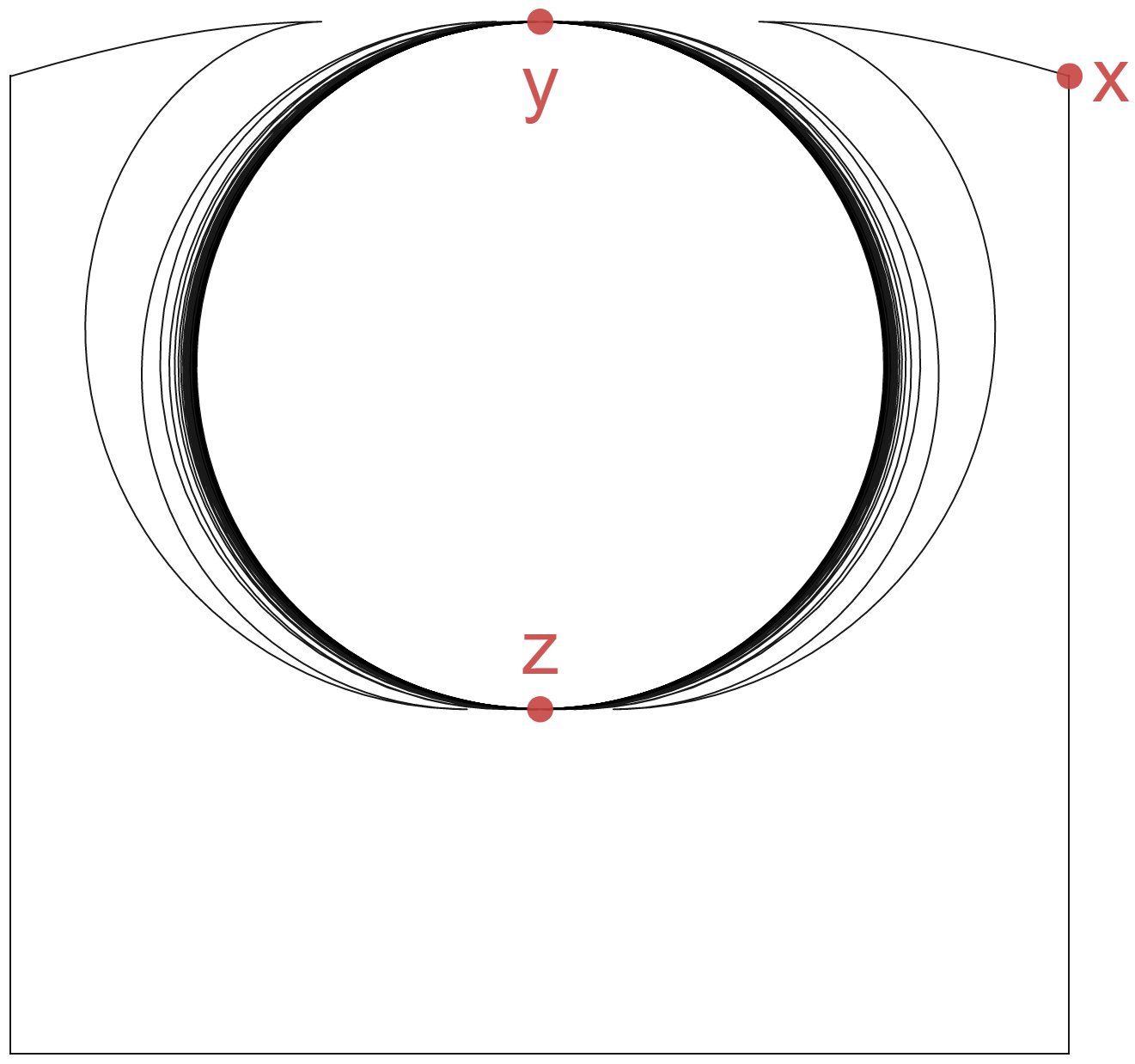}
\caption{The continuum $X$ in Example \ref{wilderexample}.}
\label{figure1}

\end{figure}

\begin{question}{\rm (See \cite[Question 5.2]{D**}.) }
Is the property of being a $D$-continuum (or a $D^{*}$-continuum, or a  $D^{**}$-continuum) a Whitney reversible property?  

\end{question}

\section{Inverse limits with monotone upper semi-continuous bonding functions}

In this section, we deal with topics on inverse  limits with set-valued bonding functions.

\vspace{1mm}

Note that if a map $f:X  \to Y$ between compacta is monotone, then it satisfies the condition of $f$ in the following definition.

\begin{definition}[{\cite[Definition~3.2]{kelly}}]
 Let $X$ and $Y$ be compacta. Then, an upper semi-continuous function $f:X\to 2^Y$ is said to be {\it monotone} if each of the projection maps $p_X^{G(f)}:G(f)\to X$ and  $p_Y^{G(f)}:G(f)\to Y$ is a monotone map.
\end{definition}

The following results are proven in \cite{oshima}.

 

\begin{theorem}{\rm (\cite[Example 5.1]{oshima})}
There exists an inverse sequence $\{X_i,f_i\}_{i=1}^{\infty}$ of $D^*$-continua with surjective monotone upper semi-continuous bonding functions such that $\underleftarrow{{\rm lim}}\{X_i,f_i\}_{i=1}^{\infty}$ is not a $D^*$-continuum.
 \label{invd*}
\end{theorem}

\begin{theorem}{\rm (\cite[Example 5.4]{oshima})}
There exists an inverse sequence $\{X_i,f_i\}_{i=1}^{\infty}$ of $D$-continua with surjective monotone upper semi-continuous bonding functions such that $\underleftarrow{{\rm lim}}\{X_i,f_i\}_{i=1}^{\infty}$ is not a $D$-continuum.
 \label{invd}
\end{theorem}

The main aim of this section is to prove Theorem \ref{inverselimitD**}, which is in contrast to the above results. To prove the theorem, we need Theorem \ref{mainoshima} proven in \cite{oshima}.


\begin{definition}{\rm (\cite[Definition 3.1]{oshima})}
Let $\mathcal{P}$ be a topological property of continua. Then, we say that  {\it property $\mathcal{P}$ satisfies the  monotonic condition I} if for any continua $X$ and $Y$ with property $\mathcal{P}$ and for  each surjective monotone upper semi-continuous function $f: X \to 2^Y$, $G(f)$ is a continuum having property $\mathcal{P}$.  
\end{definition}

\begin{definition}{\rm (\cite[Definition 3.2]{oshima})}
Let $\mathcal{P}$ be a topological property of continua. Then, we say that   {\it property $\mathcal{P}$ satisfies the  monotonic condition II} if for any inverse sequence $\{X_i,f_i\}_{i=1}^{\infty}$ of continua having property $\mathcal{P}$ with surjective monotone bonding maps, $X_{\infty}$ is a continuum having property $\mathcal{P}$.

\end{definition}


\begin{theorem} {\rm (\cite[Theorem 3.6]{oshima})} Let $\mathcal{P}$ be a topological property of continua. Then, property $\mathcal{P}$ satisfies the monotonic conditions I and II if and only if for each inverse sequence  $\{X_i,f_i\}_{i=1}^{\infty}$  of continua  having property $\mathcal{P}$ with surjective monotone upper semi-continuous functions,  $X_\infty$ is a continuum having  property $\mathcal{P}$.  
\label{mainoshima}
\end{theorem}

Furthermore, we need the following result. The proof is similar to the proof of \cite[Theorem 4.3]{D}. Hence, we omit the proof. 





\begin{lemma}{\rm (See also \cite[Theorem 4.3]{D})}
Let $\{X_i,f_i\}_{i=1}^{\infty}$ be an inverse sequence of $D^{**}$-continua with surjective monotone maps. Then, $\underleftarrow{{\rm lim}}\{X_i,f_i\}_{i=1}^{\infty}$ is also a $D^{**}$-continuum.
\label{inverselimitd**}
\end{lemma}

The following theorem is the main result in this section.

\begin{theorem}
Let $\{X_i,f_i\}_{i=1}^{\infty}$ be an inverse sequence of compacta with surjective monotone upper semi-continuous  bonding functions. If $X_i$ is a $D^{**}$-continuum for each $i \ge 1$, then $\underleftarrow{{\rm lim}}\{X_i,f_i\}_{i=1}^{\infty}$ is also a $D^{**}$-continuum.
\label{inverselimitD**}
\end{theorem}

\begin{proof}
By Theorem \ref{mainoshima}, it is enough to show that the property of being $D^{**}$ satisfies the monotonic conditions I and II. By Lemma \ref{inverselimitd**}, it is easy to see that  the property of being $D^{**}$ satisfies the monotonic condition II. Hence, we only show that the property of being $D^{**}$ satisfies the monotonic condition I. 

Let $X$ and $Y$ be $D^{**}$-continua and let  $f:X \to 2^Y$ be a surjective monotone upper semi-continuous function.  
First, since $f$ is upper semi-continuous, by Theorem \cite[Theorem~2.1]{ingram-mahavier}, $G(f)$ is compact. Also, by the monotonicity of $f$, $G(f)$ is connected. Hence, $G(f)$ is a continuum.

 Let $(x,y)\in G(f)$ and let  $A$  be a nondegenerate subcontinuum of $G(f)$ such that $(x,y) \notin A$. Put $p_X=p_X^{G(f)}$ and $p_Y=p_Y^{G(f)}$.
 
 \begin{itemize}
  \item[{\bf Case 1}]  $x \notin p_X(A)$ and $y \notin p_Y(A)$ 
  
  Since $A$ is not a one point set, we may assume $p_X(A)$ is not a one point set. Since $X$ is $D^{**}$, there exists a subcontinuum $C \subset  X$ such that $x \in C$, $p_X(A) \cap C \neq \emptyset$ and $p_X(A) \setminus C \neq \emptyset$. Then, $p_X^{-1}(C)$ is a subcontinuum of $G(f)$ such that $(x,y) \in p_X^{-1}(C)$, $p_X^{-1}(C) \cap A \neq \emptyset$ and $A \setminus p_X^{-1}(C) \neq \emptyset$.

  \medskip
  \item[{\bf Case 2.}]$x \in p_X(A)$ or $y \in p_Y(A)$. 
  
  In this case, we may assume $x \in p_X(A)$.
  
  \smallskip
  \begin{itemize}
   \item[{\bf Case 2.1.}] $p_X(A)$ is not a one point set. 
   
    In this case, by using the fact that $f$ is monotone, we see that $p_X^{-1}(x)$ is a subcontinuum of $G(f)$ such that $(x,y) \in p_X^{-1}(x)$, $p_X^{-1}(x) \cap A \neq \emptyset$ and 
    $A \setminus p_X^{-1}(x) \neq \emptyset$. 
   
    
    
    
 
   
   \smallskip
   
   \item[{\bf Case 2.2.}]$p_X(A)$ is a  one point set.
  
  In this case, note that $p_Y(A)$ is a nondegenerate subcontinuum of $Y$ and $y \notin p_Y(A)$. Since $Y$ is $D^{**}$, there exists a subcontinuum $D$ of $Y$ such that  $y \in D$, $p_Y(A) \cap D \neq \emptyset$ and $p_X(A) \setminus D \neq \emptyset$. Then, $p_X^{-1}(D)$ is a subcontinuum of $G(f)$ such that $(x,y) \in p_X^{-1}(D)$, $p_X^{-1}(D) \cap A \neq \emptyset$ and $A \setminus p_X^{-1}(D) \neq \emptyset$.

  \end{itemize}
 
 \end{itemize}
 
From the above, we see that $G(f)$ is a $D^{**}$-continuum. Hence, the property of being $D^{**}$ satisfies the monotonic condition I. This completes the proof. \end{proof}

\begin{example}
 Let $A_1$, $A_2$, $A_3$ and $A_4$ be subspaces of $\mathbb{R}^2$ defined by
 $$
 \begin{array}{ccl}
     A_1 & = & \{({\rm cos}(t)+3,{\rm sin}(t)-2)\in \mathbb{R}^2~|~0\leq t\leq 2\pi\},\\
     A_2 & = & \{(t+|{\rm cos}(\frac{1}{t})|+3,-{\rm sin}(\frac{1}{t})-2)\in \mathbb{R}^2~|~0<t\leq 1\}\\
         &   & \cup ([{\rm cos}(1)+4,5]\times \{-{\rm sin}(1)-2\})\\
         &   & \cup \{(-t+6,-{\rm sin}(\frac{1}{t})-2)\in \mathbb{R}^2~|~0<t\leq 1\}\\
         &   & \cup (\{6\}\times [-3,{\rm sin}(\frac{1}{6})+2)),\\
     A_3 & = & \{(6,{\rm sin}(\frac{1}{6})+2)\}, {\rm and}\\
     A_4 & = & \{(t,{\rm sin}(\frac{1}{t})+2)\in \mathbb{R}^2~|~0<t<6\} \cup (\{0\}\times [-{\rm sin}(1)-2,3])\\
         &   & \cup ([0,-{\rm cos}(1)+2]\times \{-{\rm sin}(1)-2\})\\
         &   & \cup \{(-t-|{\rm cos}(\frac{1}{t})|+3,-{\rm sin}(\frac{1}{t})-2)\in \mathbb{R}^2~|~0<t\leq 1\}.
 \end{array}
 $$
 Note that $A_1,~A_2,~A_3,~{\rm and}~A_4$ are pairwise disjoint. Let $X=A_1\cup A_2\cup A_3\cup A_4$ (see figure \ref{continuumxx}). We can see that there is a homeomorphism $h:A_2\cup A_4\to A_2\cup A_4$ such that $h(A_2)=A_4~{\rm and}~h(A_4)=A_2$. 

 \begin{figure}
\includegraphics[scale=0.23]{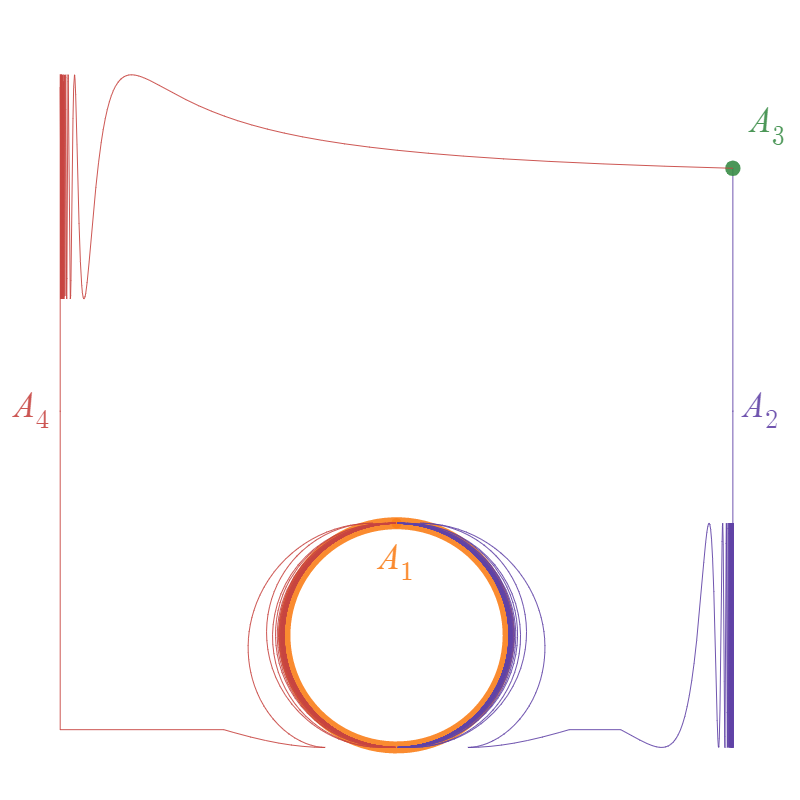}
\caption{The continuum $X$ in Example \ref{continuumxx}}
\label{continuumxx}
\end{figure}


Let $f:X\to 2^{X}$ be the upper semi-continuous function defined by
\begin{align}
 f(x)&=
\begin{cases}
~A_3&(x\in A_1),\\
~\{h(x)\}&(x\in A_2),\\
~A_1&(x\in A_3),\\
~\{h(x)\} &(x\in A_4).
\end{cases}\nonumber
\end{align}

It is easy to see that $f$ is a surjective monotone upper semi-continuous function.

For each $i\geq 1$, let $X_i = X$ and let $f_i=f:X_{i+1}\to 2^{X_i}$. By Theorem \ref{inverselimitD**}, the inverse limit $\underleftarrow{{\rm lim}}\{X_i,f_i\}_{i=1}^{\infty}$ is a $D^{**}$-continuum.
\label{continuumxx}
\end{example}

\begin{remark}
 The continuum $X$ in the previous example also appears in \cite[Example 2.3]{D**} as an example of $D^{**}$-continuum which is neither Wilder nor $D^{*}$.   
\end{remark}

\begin{remark}In \cite{oshima}, the third author showed that inverse limits of Wilder continua with surjective monotone upper semi-continuous bonding functions are also Wilder continua. 
\end{remark}

\section{A $D^{**}$-continuum containing neither Wilder continua nor $D^*$-continua}
Let $X$ be a nondegenerate continuum. Then $X$ is said to be {\it arc-like} if for each $\varepsilon > 0$, there exists a surjective map $f : X \to [0,1]$ such that for each $y \in [0,1]$, ${\rm diam} f^{-1}(y) < \varepsilon$. A continnuum is said to be $hereditarily$ $decomposable$ if each of its all subcontinua is decomposable.

\begin{theorem}{\rm (\cite[Remark 5,11]{D})}
There exists an arc-like hereditarily decomosable continuum containing no $D$-continua.    
\end{theorem}

    A continuum  $X$ is said to be {\it hereditarily arcwise connected} if each of its subcontinuum is arcwise connected. Also, if each subcontinuum of $X$ is  a Wilder continuum (resp.   a $D$-continuum,  a $D^*$-continuum, a $D^{**}$-continuum), then $X$ is called  {\it a hereditarily Wilder continuum} (resp. {\it  a hereditarily $D$-continuum, a hereditarily $D^*$-continuum, a hereditarily $D^{**}$-continuum}).

\begin{theorem}{\rm (See \cite[Cororally 5,12]{D}) and \cite[Cororally 3.2]{D**})}
There exists an arc-like hereditarily $D$-continuum containing no $D^{**}$-continua.    
\end{theorem}

Hence, it is natural to ask whether or not there exists a hereditarily $D^{**}$-continuum containing no $D^*$-continua.    The answer to this question is negative by the following result.

\begin{corollary}{\rm \cite[Corollary 3.4]{D**})}
Let $X$ be a nondegenerate continuum. Then, the following are equivalent. 

\begin{itemize}
    \item[(1)] $X$ is a hereditarily Wilder continuum;
    \item[(2)] $X$ is a hereditarily $D^*$-continuum;
    \item[(3)] $X$ is a hereditarily $D^{**}$-continuum; 
    \item[(4)] $X$ is a hereditarily arcwise connected continuum.
    
\end{itemize}
\label{corhere}
\end{corollary}

However, in Theorem \ref{mainex}, we show that there exists a $D^{**}$-continuum containing neither Wilder continua nor $D^{*}$-continua. 
\vspace{1mm}

Let $B$ be the subcontinuum of $\mathbb{R}^2$ defined by
$$
\begin{array}{ccl}
    B & = & (\{\frac{16}{\pi}+1\} \times [-(\sin{(1)}+2),3])\\
      &   & \cup \{(t+|\cos{(\frac{4}{t})}|+\frac{8}{\pi}+1,\sin{(\frac{4}{t})}+2) \in \mathbb{R}^2 \ | \ 0 < t \leq \frac{8}{\pi} \}\\
      &   & \cup \{(\cos{(t)}+\frac{8}{\pi}+1,\sin{(t)}+2) \in \mathbb{R}^2 \ | \ 0 \leq t \leq 2\pi\}\\
      &   & \cup \{(-t-|\cos{(\frac{4}{t})}|+\frac{8}{\pi}+1,\sin{(\frac{4}{t})}+2) \in \mathbb{R}^2 \ | \ 0 < t \leq \frac{8}{\pi}\}\\
      &   & \cup ([-1,1] \times \{3\})\\
      &   & \cup (\{-1\} \times [1,3]) \cup \{(-(t+1),\sin{(\frac{16}{\pi t})}+2) \in \mathbb{R}^2 \ | \ 0 < t \leq \frac{16}{\pi}\}\\
      &   & \cup (\{-(\frac{16}{\pi}+1)\} \times [-3,\sin{(1)}+2])\\
      &   & \cup \{(-(t+|\cos{(\frac{4}{t})}|+\frac{8}{\pi}+1),-(\sin{(\frac{4}{t})}+2)) \in \mathbb{R}^2 \ | \ 0 < t \leq \frac{8}{\pi} \}\\
      &   & \cup \{(\cos{(t)}-(\frac{8}{\pi}+1),\sin{(t)}-2) \in \mathbb{R}^2 \ | \ 0 \leq t \leq 2\pi\}\\
      &   & \cup \{(-(-t-|\cos{(\frac{4}{t})}|+\frac{8}{\pi}+1),-(\sin{(\frac{4}{t})}+2)) \in \mathbb{R}^2 \ | \ 0 < t \leq \frac{8}{\pi}\}\\
      &   & \cup ([-1,1] \times \{-3\})\\
      &   & \cup (\{1\} \times [-3,-1]) \cup \{(t+1,-(\sin{(\frac{16}{\pi t})}+2)) \in \mathbb{R}^2 \ | \ 0 < t \leq \frac{16}{\pi}\}
\end{array}
$$
(see Figure \ref{continuumx}). Let $P$ and $Q$ be the subcontinua of $B$ defined by
$$
\begin{array}{ccl}
    P & = & (\{\frac{16}{\pi}+1\} \times [0,3])\\
      &   & \cup \{(t+|\cos{(\frac{4}{t})}|+\frac{8}{\pi}+1,\sin{(\frac{4}{t})}+2) \in \mathbb{R}^2 \ | \ 0 < t \leq \frac{8}{\pi} \}\\
      &   & \cup \{(\cos{(t)}+\frac{8}{\pi}+1,\sin{(t)}+2) \in \mathbb{R}^2 \ | \ 0 \leq t \leq 2\pi\}\\
      &   & \cup \{(-t-|\cos{(\frac{4}{t})}|+\frac{8}{\pi}+1,\sin{(\frac{4}{t})}+2) \in \mathbb{R}^2 \ | \ 0 < t \leq \frac{8}{\pi}\}\\
      &   & \cup ([-1,1] \times \{3\})\\
      &   & \cup (\{-1\} \times [1,3]) \cup \{(-(t+1),\sin{(\frac{16}{\pi t})}+2) \in \mathbb{R}^2 \ | \ 0 < t \leq \frac{16}{\pi}\}\\
      &   & \cup (\{-(\frac{16}{\pi}+1)\} \times [0,\sin{(1)}+2]),\\
      
    Q & = & (\{-(\frac{16}{\pi}+1)\} \times [-3,0])\\
      &   & \cup \{(-(t+|\cos{(\frac{4}{t})}|+\frac{8}{\pi}+1),-(\sin{(\frac{4}{t})}+2)) \in \mathbb{R}^2 \ | \ 0 < t \leq \frac{8}{\pi} \}\\
      &   & \cup \{(\cos{(t)}-(\frac{8}{\pi}+1),\sin{(t)}-2) \in \mathbb{R}^2 \ | \ 0 \leq t \leq 2\pi\}\\
      &   & \cup \{(-(-t-|\cos{(\frac{4}{t})}|+\frac{8}{\pi}+1),-(\sin{(\frac{4}{t})}+2)) \in \mathbb{R}^2 \ | \ 0 < t \leq \frac{8}{\pi}\}\\
      &   & \cup ([-1,1] \times \{-3\})\\
      &   & \cup (\{1\} \times [-3,-1]) \cup \{(t+1,-(\sin{(\frac{16}{\pi t})}+2)) \in \mathbb{R}^2 \ | \ 0 < t \leq \frac{16}{\pi}\}\\
      &   & \cup (\{\frac{16}{\pi}+1\} \times [-(\sin{(1)}+2),0]).
\end{array}
$$
Also, let $R^{0}$, $Q^{0}$,
$R^{1}$ and $Q^{1}$ be the subspaces of $B$ defined by 
$$
\begin{array}{ccl}
    R^{0} & = & \{(-(t+1),\sin{(\frac{16}{\pi t})}+2) \in \mathbb{R}^2 \ | \ 0 < t \leq \frac{16}{\pi}\},\\
    
    Q^{0} & = & \{(-t-|\cos{(\frac{4}{t})}|+\frac{8}{\pi}+1,\sin{(\frac{4}{t})}+2) \in \mathbb{R}^2 \ | \ 0 < t \leq \frac{8}{\pi}\},\\
    
    R^{1} & = & \{(t+1,-(\sin{(\frac{16}{\pi t})}+2)) \in \mathbb{R}^2 \ | \ 0 < t \leq \frac{16}{\pi}\},\\

    Q^{1} & = & \{(t+|\cos{(\frac{4}{t})}|-\frac{8}{\pi}-1,-(\sin{(\frac{4}{t})}+2)) \in \mathbb{R}^2 \ | \ 0 < t \leq \frac{8}{\pi}\}.
\end{array}
$$
  Then, $R^{0}$, $Q^{0}$,
$R^{1}$ and $Q^{1}$ are homeomorphic to $[0,1)$ respectively. Also, note that ${\rm Cl}_B R^{0}$, ${\rm Cl}_B Q^{0}$, ${\rm Cl}_B R^{1}$ and ${\rm Cl}_B Q^{1}$ are homeomorphic to the $\sin \frac{1}{x}$-continuum.  Furthermore, let $a$, $b$, $p$, $p'$, $q$, $q'$, $s$, $s'$, $t$, $t'$, $u$, $u'$, $v$, $v'\in B$ be the points as in Figure \ref{continuumx}. Finally, let $S(B)=\{a$, $b$, $p$, $p'$, $q$, $q'$, $s$, $s'$, $t$, $t'$, $u$, $u'$, $v$, $v'\}$.

\smallskip

 In the following theorem, if $\{X_i, f_i\}_{i = 0}^{\infty}$ is an inverse sequence with single valued bonding maps, then we denote the map $f_i \circ \cdots \circ f_{j-1}: X_j \to {X_i}$ by $f_{i,j
 }$ if $j > i +1$.  Also, let $f_{i,i+1}=f_i$.  Furthermore, for each $n \ge 0$, let  $\pi_n : X_\infty \to X_n$ be the projection.  Finally, if $f:X \to Y$ is a map and $y \in Y$ is a point such that $f^{-1}(y)$ is a one point set, then we also denote the point in $f^{-1}(y)$ by $f^{-1}(y)$.

\begin{theorem} \label{mainex}  \label{mainex} 
There exists a $D^{**}$-continuum containing neither Wilder continua nor $D^{*}$-continua.
\end{theorem}
    
\begin{proof}

We follow the scheme from  \cite[Section 3]{hago3}, which is originally based on the idea of the construction of Janiszewski's hereditarily decomposable, arc-like, and arcless continuum \cite{jani}. In particular, almost all notations and terminology in the proof are taken from \cite[Section 3]{hago3}. In order to avoid complications, we will omit the parts of the way of construction of the required continuum that overlap with the construction of the planar arcless continuum-chainable continuum that appeared in \cite{hago3}. There are three differences in construction between the continuum in \cite[Section 3]{hago3} and ours. Two of them are as follows:

\smallskip
(i) Instead of $B$ in \cite[Section 3]{hago3}, in our example, 
we let $B$ be the continuum defined before Theorem \ref{mainex}.

\smallskip

(ii) Let  $X_0= \{(x,0) \in \mathbb{R}^2\ |  \ -1 \le x \le 1\}$ and $d_0=(0,0)$. Also, let $X_1 = B \cup  \{(x,0) \in \mathbb{R}^2\ |  \  -\frac{16}{\pi}-2  \le x \le -\frac{16}{\pi}-1  \ {\rm or } \ \frac{16}{\pi}+1 \le x \le \frac{16}{\pi}+2\}$ and let  $f_0^{1} : X_1 \to X_0$ be the map defined by

$$f_0^1(x,y)=\left\{
\begin{array}{l}
(x+\frac{16}{\pi}+1,y) \hspace{1pc} (x \le -\frac{16}{\pi}-1  ), \\
(0,0) \hspace{1pc} (-\frac{16}{\pi}-1 \le x \le \frac{16}{\pi}+1), \\
(x-\frac{16}{\pi}-1,y) \hspace{0.8pc} (x \ge \frac{16}{\pi}+1)).
\end{array}
\right.$$




\smallskip

See Figure \ref{???}. 
 We  say that $X_1$ is obtained from $X_0$
by blowing up the point $d_0 = (0,0) \in X_0$   to a copy of $B$. Also, following the scheme in  \cite[Section 3]{hago3}, we continue blowing up points of an appropriately chosen sequence $\{d_n\}_{n=0}^\infty$ to copies of $B$, and for each $n \ge 0,$ we can obtain $X_{n+1}$ from $X_n$ and $f_n:X_{n+1} \to X_n$. For each $n \ge 1$, let $B_n=f_{n-1}^{-1}(d_{n-1})$ and let $P_n$, $Q_n$,  $R_n^{0}$,  $Q_n^0$, $R_n^{0}$,   $R_n^1$,   $a_n$,  $b_n$, $p_n$, $p'_n$,  $q_n$,  $q'_n$,  $s_n$,  $s_n'$, $t_n$,  $t_n'$, $u_n$,  $u_n'$,  $v_n$,  $v_n'$ be subsets and points in $B_n$ corresponding to $P$, $Q$,  $R^{0}$,   $Q^0$,  $R^{1}$, $Q^1$,  $a$,  $b$,   $p$, $p'$, $q$,   $q'$,  $s$,  $s'$,  $t$,  $t'$,   $u$,  $u'$,  $v$,  $v'$ in $B$ respectively. Let $S(X_0)= E(X_0)$. Also, for each $n \ge 1$, 
let $$S(X_n)=f_{0,n}^{-1}(S(X_0)) \cup (\bigcup_{k=1}^n f_{k,n}^{-1}(\{ a_k , b_k, p_k, p'_k,  q_k,  q'_k,  s_k,  s_k', t_k, t_k',  u_k,  u_k',  v_k,  v_k'\}))$$.

We may assume  the following condition (this is the third difference in construction between the planar arcless continuum-chainable continuum that appeared in \cite[Section 3]{hago3}) and ours):

\smallskip

(iii) For each $n \ge 1$, $d_n \in X_n \setminus S(X_n)$.

\smallskip
Let $X_\infty= \underleftarrow{{\rm lim}}\{X_n,f_n\}_{n=0}^{\infty}$. We show that $X_\infty$ has the required property. 

\smallskip

{\bf Claim 1.} $X_\infty$ is a $D^{**}$-continuum.

Proof of Claim 1. To prove Claim 1, first, we prove that $X_n$ is a $D^{**}$-continuum for each $n \ge 0$. To prove it , we need to prove the following subclaim.

\smallskip

{\bf Subclaim.} Let $n \ge 0$. Let $z \in X_n \setminus S(X_n)$ and $A$ be a subarc of $X_n$ such that $A \neq X_n$ and $z \in A \setminus E(A)$. Then, for each $x \in X_{n} \setminus A$, there exists a subcontinuum $C$ of $X_n$ such that $x \in C, \ A \cap C \neq \emptyset$ and $z \notin C$.     

\smallskip

Before the proof of Subclaim, note that the following holds:

\smallskip

$(\star)$ Let $b \in B \setminus S(B)$ and let $L$ be a subarc of $B$ such that  $b \in L \setminus E(L)$. Then, for each $y \in B \setminus L$, there exists a subcontinuum $D$ of $B$ such that $y \in D, \ L \cap D \neq \emptyset$ and $b \notin D$.   

\smallskip

Proof of Subclaim.  We prove Subclaim  by induction on $n$. If $n=0$, then we can easily see the assertion holds. Let $k \ge 1$ and assume that the assertion holds for $ n \le k-1$.   Now, we prove that the assertion holds for $n=k$.

\smallskip

{\bf Case A.} $z, x \notin B_k$. 

{\bf Case B.} $z \notin B_k, \ x \in B_k$.

In the cases above, by using inductive assumption, we can easily find a subcontinuum $C$ of $X$ with the required properties.  

\smallskip

{\bf Case C.} $z, x \in B_k$. 

In this case, by ($\star$) it is easy to see that there is a subcontinuum $C$ of $X$ with the required properties. 

\smallskip

{\bf Case D.} $z \in B_k$, $x \notin B_k$ . 

In this case, take a sufficiently small subarc $J$ of $X_{k-1}$ such that $d_{k-1} (=f_{k-1}(B_k)) \in J \setminus E(J)$ and $f_{k-1}(x) \notin J$. Then, by inductive assumption, there is a subcontinuum $C'$ of $X_{k-1}$ such that $f_{k-1}(x) \in C', \ J \cap C' \neq \emptyset$ and $d_{k-1} \notin C'$. Take $c \in  J \cap C'$. Then, we can take a subarc $P$ of $X_{k}$ from $f_{k-1}^{-1}(c)$ to the one of $a_k$ and $b_k$, and $(P \setminus E(P)) \cap B_k = \emptyset$. We may assume $a_k \in P$. Let $L'$ be a subarc of $B_k$ such that $z \in L' \setminus E(L') \subseteq L' \subseteq A \setminus E(A)$ and $a_k \notin L'$. Then, by $(\star)$, there exists a subcontinuum $D'$ of $B_k$ such that $a_k \in D'$, $L' \cap D' \neq \emptyset$ and $z \notin D'$. Let  $C=f_{k-1}^{-1}(C') \cup P \cup D'$. Then, it is easy to see that $C$ has the required properties.

\smallskip

Hence, in every case we can find a subcontinuum $C$ of $X_k$ with the required properties. This completes the proof of Subclaim.

\smallskip

Now, we prove $X_n$ is a $D^{**}$-continuum for each $n \ge 0$. 
Let $F \subsetneq X_n$ be a nondegenerate subcontinuum and $x \in X_n \setminus F$. Take any subarc $A$ of $F$ and any $z \in A \setminus (E(A) \cup S(X_n))$. By Subclaim, there exists a subcontinuum $C$ of $X_n$ such that $x \in C, \ A \cap C \neq \emptyset$ and $z \notin C$. Note that $F \setminus C \neq \emptyset$. Hence, we see that $X_n$ is a $D^{**}$-continuum.

Therefore, $X_\infty$ is an inverse limit of $D^{**}$-continua with surjective monotone bonding maps. Hence, by Lemma \ref{inverselimitd**}, $X_\infty$ is a $D^{**}$-continuum. This completes the proof of Claim 1. 

\smallskip

\begin{figure}
\includegraphics[scale=0.25]{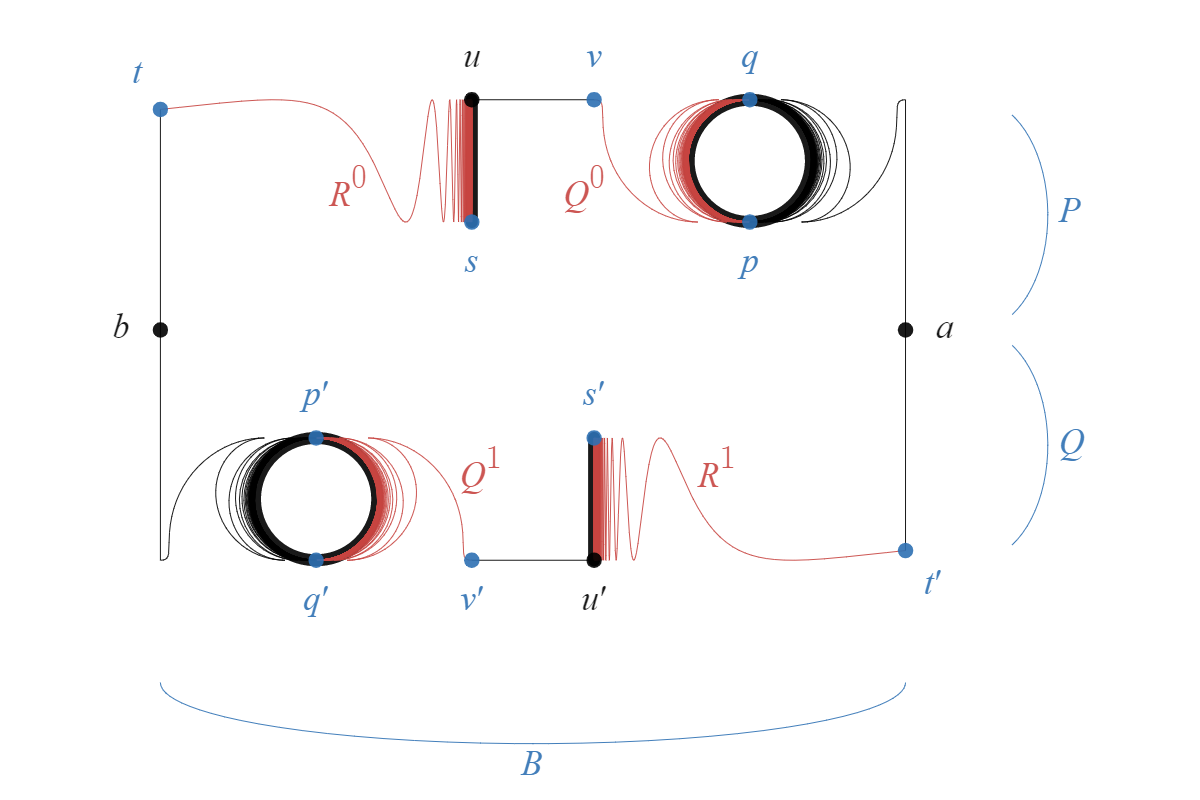}
\caption{The continuum $B$}
\label{continuumx}
\end{figure}

\begin{figure}
\includegraphics[scale=0.28]{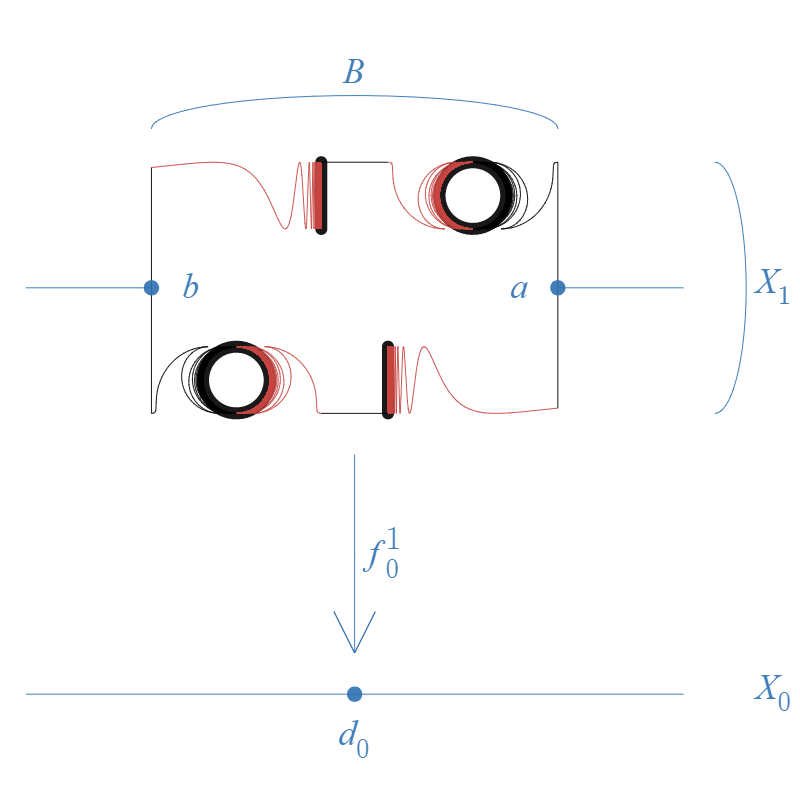}
\caption{$X_0,\ X_1 \ {\rm and} \ f_0^1.$}
\label{???}
\end{figure}

{\bf Claim 2.} $X_{\infty}$ contains no Wilder continua. 

Proof of Claim 2. Let $D$ be a nondegenerate subcontinuum of $X_{\infty}$. Then, there exists $n_0 \ge 0$ such that $\pi_{n_0}(D)$ is not a singleton. Take any subarc $A$ of $\pi_{n_0}(D) \setminus S(X_{n_0})$.  Let $n_1 = \min \{ n  \ | \ n \ge n_0$ and  $d_n \in f_{n_0,n}^{-1} (A \setminus E(A)) \}$ and $n_2 = \min \{ n  \ | \ n > n_1     \ {\rm and} \    d_n \in f_{n_1,n}^{-1}(f_{n_0,n_1}^{-1}(A \setminus E(A)) \setminus \{d_{n_1}\})  \}$. Note that $f_{n_1+1, n_2+1}^{-1}(a_{n_1+1})$, $f_{n_1+1, n_2+1}^{-1}(b_{n_1+1})$, $a_{n_2+1}$, $b_{n_2+1} \in \pi_{n_2+1}(D)$. Since  $\pi_{n_2+1}(D)$ is connected, the one of the following holds:

\smallskip

(i) There is a subcontinuum $T$ of $f_{n_1, n_2+1}^{-1}(d_{n_1})$ containing  $f_{n_1+1, n_2+1}^{-1}(a_{n_1+1})$ and  $f_{n_1+1, n_2+1}^{-1}(b_{n_1+1})$.

\vspace{1mm}

(ii) There is a subcontinuum of $T' $ of $f_{n_2}^{-1}(d_{n_2})$ containing  $a_{n_2+1}$ and  $b_{n_2+1}$.

\smallskip 

Assume (i) holds. Then, it is easy to see that either $\{p_{n_1+1}, q_{n_1+1}\} \subseteq f_{n_1+1, n_2+1}(T)$ or $\{p'_{n_1+1}, q'_{n_1+1}\} \subseteq f_{n_1+1, n_2+1}(T)$ holds. Hence, we may assume $\{p_{n_1+1}, q_{n_1+1}\} \subseteq f_{n_1+1, n_2+1}(T)$. In this case, we can easily see that $\pi_{n_1+1}^{-1}(p_{n_1+1})$ and $\pi_{n_1+1}^{-1}(q_{n_1+1})$ are points of non-Wilderness of $D$. Hence, if (i) holds, then $D$ is not a Wilder continuum. We also see that $D$ is not a Wilder continuum if (ii) holds. This completes the proof of Claim 2.

\vspace{2mm}

{\bf Claim 3.}  $X_{\infty}$ contains no $D^{*}$-continua.

Proof of Claim 3. Let $D, \  n_0, \ n_1, \ n_2$ be as in the proof of Claim 2.  Then, as in the proof of Claim 2,  the one of the following holds:

\vspace{1mm}

(i) There is a subcontinuum $T$ of $f_{n_1, n_2+1}^{-1}(d_{n_1})$ containing  $f_{n_1+1, n_2+1}^{-1}(a_{n_1+1})$ and  $f_{n_1+1, n_2+1}^{-1}(b_{n_1+1})$.

\vspace{1mm}

(ii) There is a subcontinuum of $T' $ of $f_{n_2}^{-1}(d_{n_2})$ containing  $a_{n_2+1}$ and  $b_{n_2+1}$.

\vspace{1mm}

Assume (i) holds. We may assume that 
$P_{n_1+1} \subseteq f_{n_1+1, n_2+1}(T)$.  Let $R'={\rm Cl}_{X_\infty}(\pi_{n_1+1}^{-1}(R_{n_1+1}^0)) \setminus \pi_{n_1+1}^{-1}(R_{n_1+1}^0)$ and  $Q'={\rm Cl}_{X_\infty}(\pi_{n_1+1}^{-1}(Q_{n_1+1}^0)) \setminus \pi_{n_1+1}^{-1}(Q_{n_1+1}^0)$. Then, $R'$ and $Q'$ are pairwise disjoint subcontinua of $X_{\infty}$.
 Also, it is not difficult to see that if $Z$ is a subcontinuum of $X_\infty$ such that $R' \cap Z \neq \emptyset \neq Q' \cap Z$ and $R' \setminus Z \neq \emptyset \neq Q' \setminus Z$, then $\pi_{n_1+1}(Z)$ is a subcontinuum of $X_{n_1+1}$ such that $({\rm Cl}_{X_{n_1+1}}(R_{n_1+1}^0) \setminus R_{n_1+1}^0 ) \cap \pi_{n_1+1}(Z) \neq \emptyset \neq ({\rm Cl}_{X_{n_1+1}}(Q_{n_1+1}^0) \setminus Q_{n_1+1}^0) \cap \pi_{n_1+1}(Z) $ and $({\rm Cl}_{X_{n_1+1}}(R_{n_1+1}^0) \setminus R_{n_1+1}^0 ) \setminus \pi_{n_1+1}(Z)  \neq \emptyset \neq ({\rm Cl}_{X_{n_1}+1}(Q_{n_1+1}^0) \setminus Q_{n_1+1}^0 ) \setminus \pi_{n_1+1}(Z)  $.  However, there is not such a subcontinuum in $X_{n_1+1}$. Hence, this is a contradiction. Therefore,  there is not a subcontinuum $Z$ of $X_\infty$ such that $R' \cap Z \neq \emptyset \neq Q' \cap Z$ and $R' \setminus Z \neq \emptyset \neq Q' \setminus Z$. Hence, if (i) holds, then $D$ is not a $D^{*}$-continuum. We also see that $D$ is not a $D^{*}$-continuum if (ii) holds. This completes the proof of Claim 3.

 \vspace{2mm}

 Thus, $X_\infty$ is a $D^{**}$-continuum containing neither Wilder continua nor $D^{*}$-continua.  \end{proof}

Also, we can get the following result. The proof is very similar to the proof of Theorem \ref{mainex}. The essential difference is the continuum obtained by blowing up a point of the factor space at each step. Hence, we only write about it in the proof.

\begin{theorem}
There exists a Wilder continuum which contains no $D^{*}$-continua. 
\label{wilder}
\end{theorem}

\begin{proof}
The proof of this theorem can be easily obtained by replacing the corresponding continua drawn in Figures \ref{continuumx} and \ref{???} with the continua drawn in Figures \ref{continuumwilder} and \ref{continuumwilder2} in the proof of the previous theorem.
\end{proof}
\begin{figure}
\includegraphics[scale=0.25]{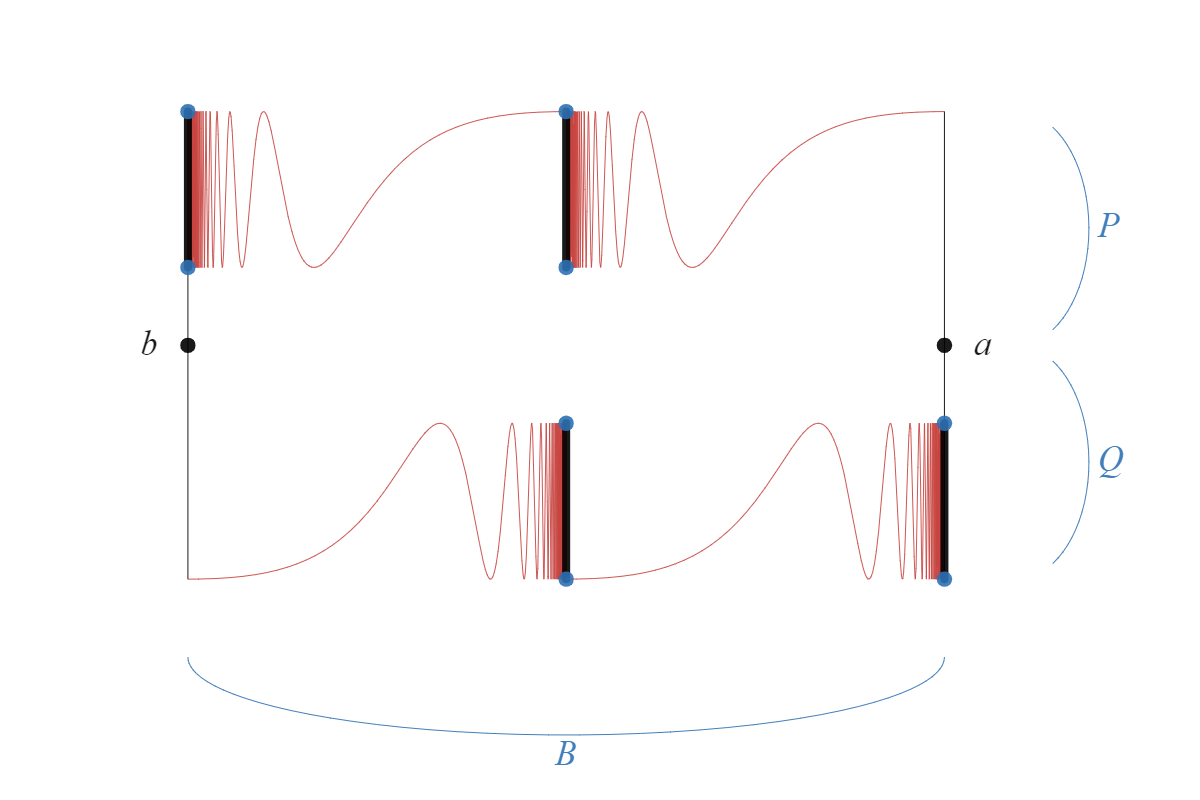}
\caption{The continuum $B$}
\label{continuumwilder}
\end{figure}

\begin{figure}
\includegraphics[scale=0.28]{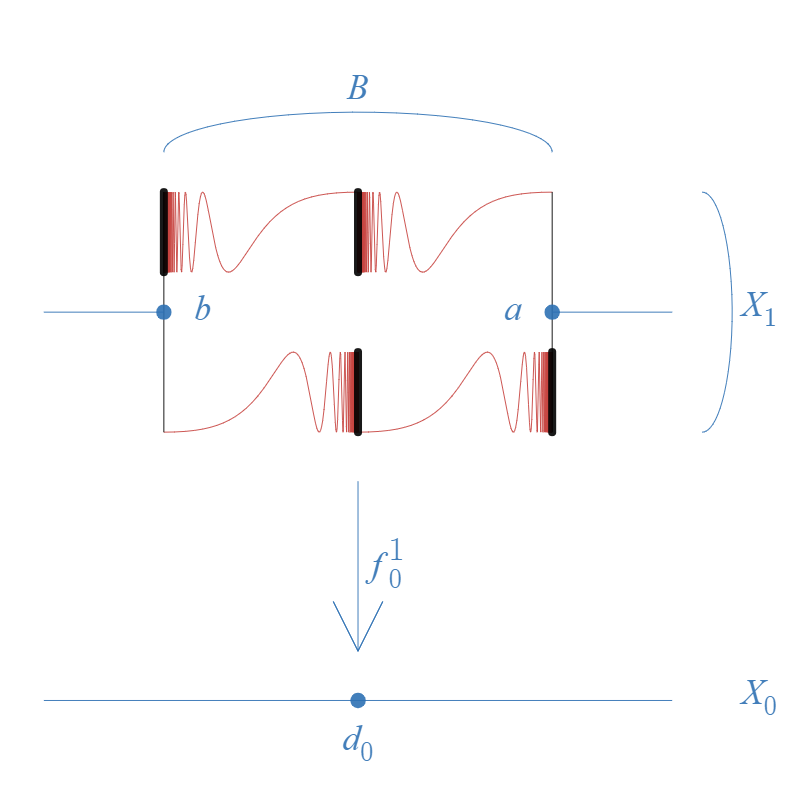}
\caption{$X_0,\ X_1 \ {\rm and} \ f_0^1.$}
\label{continuumwilder2}
\end{figure}

Furthermore, we can get the following result. 

\begin{theorem}
There exists a $D^{*}$-continuum which contains no Wilder continua. 
\label{nowilder}
\end{theorem}

\begin{proof}
As in the previous theorem, the proof of this theorem can be obtained by replacing the corresponding continua drawn in  Figures \ref{continuumx} and \ref{???} with  the continua drawn in Figures \ref{continuumwilder3} and \ref{continuumwilder4} in the proof of Theorem \ref{mainex}. 
To be exact, let 

$$
\begin{array}{ccl}
    T_1 & = & \{(t+|\cos{(\arcsin{(\frac{1}{1+t}(\sin{(\frac{1}{t})}+t))})}|,\frac{1}{1+t}(\sin{(\frac{1}{t})}+t)) \in \mathbb{R}^2 \ | \ 0 < t \leq 1\},\\

    p_1 & = & (1+|\cos{(\arcsin{(\frac{1}{2}\sin{(1)}+\frac{1}{2})})}|,\frac{1}{2}\sin{(1)}+\frac{1}{2}),\\

    T_2 & = & \{(t+|\cos{(\arcsin{(\frac{1}{1+t}(\sin{(\frac{1}{t})}-t))})}|,\frac{1}{1+t}(\sin{(\frac{1}{t})}-t)) \in \mathbb{R}^2 \ | \ 0 < t \leq 1\},\\

    p_2 & = & (1+|\cos{(\arcsin{(\frac{1}{2}\sin{(1)}-\frac{1}{2})})}|,\frac{1}{2}\sin{(1)}-\frac{1}{2}),\\

    R(\theta) & = & \begin{pmatrix} \cos{(\theta)} & -\sin{(\theta)} \\ \sin{(\theta)} & \cos{(\theta)} \end{pmatrix},\\

    pq & = & \{(1-\lambda) p + \lambda q \ | \ 0 \leq \lambda \leq 1\} \ ({\rm for \ each} \ p,q \in \mathbb{R}^2),\\

    S^1 & = & \{(\cos{(t)},\sin{(t)}) \in \mathbb{R}^2 \ | \ 0 \leq t \leq 2\pi\}, \ {\rm and}
\end{array}
$$

$$
\begin{array}{ccl}
    B & = & (6,4)(R(\frac{\pi}{2})p_2+(0,4)) \cup (R(\frac{\pi}{2})T_2+(0,4))\\
      &   & \cup (6,4)(R(-\frac{\pi}{2})p_1+(0,4)) \cup (R(-\frac{\pi}{2})T_1+(0,4))\\
      &   & \cup (S^1+(0,4))\\
      &   & \cup (R(\frac{\pi}{2})T_1+(0,4)) \cup (R(\frac{\pi}{2})p_1+(0,4))(-6,4)\\
      &   & \cup (R(-\frac{\pi}{2})T_2+(0,4)) \cup (R(-\frac{\pi}{2})p_2+(0,4))(-6,4)\\
      &   & \cup (-6,4)(-6,-4)\\
      &   & \cup (-6,-4)(R(-\frac{\pi}{2})p_2+(0,-4)) \cup (R(-\frac{\pi}{2})T_2+(0,-4))\\
      &   & \cup (-6,-4)(R(\frac{\pi}{2})p_1+(0,-4)) \cup (R(\frac{\pi}{2})T_1+(0,-4))\\
      &   & \cup (S^1+(0,-4))\\
      &   & \cup (R(-\frac{\pi}{2})T_1+(0,-4)) \cup (R(-\frac{\pi}{2})p_1+(0,-4))(6,-4)\\
      &   & \cup (R(\frac{\pi}{2})T_2+(0,4)) \cup (R(\frac{\pi}{2})p_2+(0,-4))(6,-4)\\
      &   & \cup (6,-4)(6,4).
\end{array}
$$

Then, using the same idea as the proof of Theorem \ref{mainex}, we can prove this theorem.
\end{proof}

\begin{figure}[H]
\includegraphics[scale=0.31]{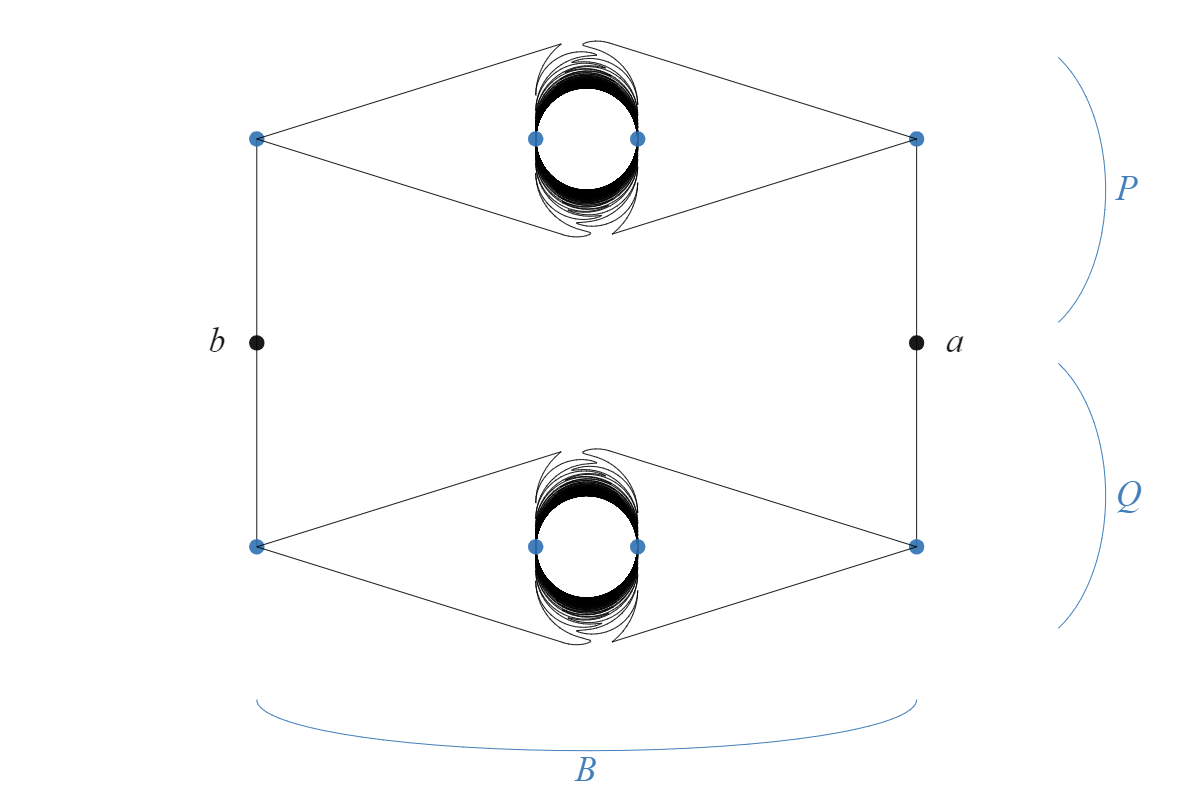}
\caption{The continuum $B$}
\label{continuumwilder3}
\end{figure}

\begin{figure}[H]
\hspace{2pc}\includegraphics[scale=0.28]{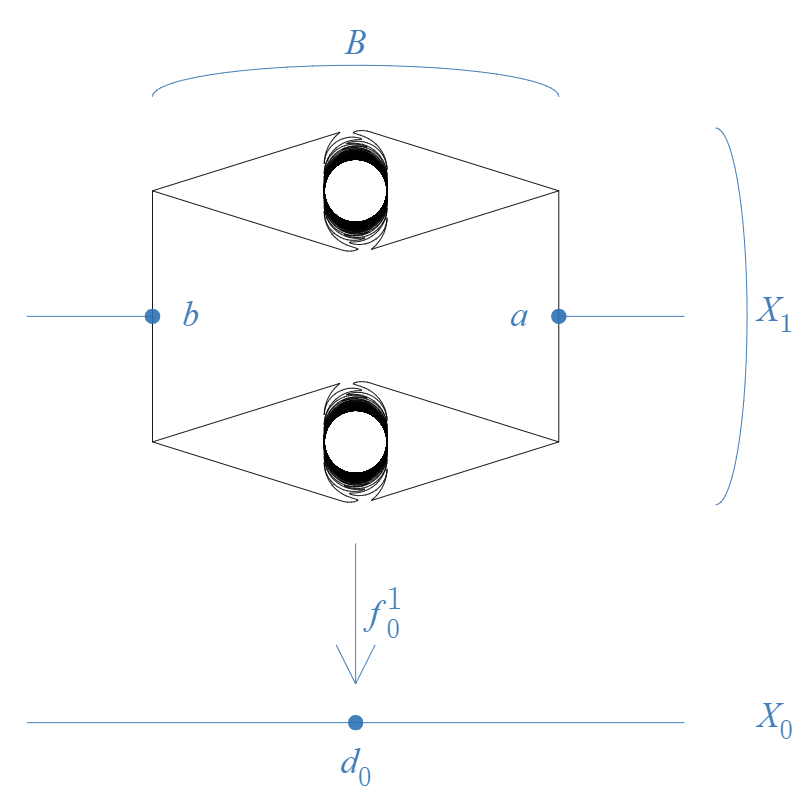}
\caption{$X_0,\ X_1 \ {\rm and} \ f_0^1.$}
\label{continuumwilder4}
\end{figure}

\end{document}